\documentclass[10pt]{amsart}
\usepackage{geometry}                
\usepackage[parfill]{parskip}    
\usepackage{amssymb, amsthm, amsmath}
\usepackage{mathtools}
\usepackage[pdftex]{graphicx}
\usepackage{xypic}
\usepackage{tikz}
\usepackage{color}
\usepackage{hyperref}
\usepackage{enumitem}
\usepackage{subcaption}

\usepackage{thmtools}
\usepackage{thm-restate}

\usepackage{multicol}
 \usepackage{booktabs}
 \newcommand{\ra}[1]{\renewcommand{\arraystretch}{#1}}

\makeatletter
\@namedef{subjclassname@1991}{2020 Mathematics Subject Classification}
\makeatother

\newtheorem{theorem}{Theorem}[section]
\newtheorem*{theorem*}{Theorem}
\newtheorem{lemma}[theorem]{Lemma}

\newtheorem{question}[theorem]{Question}

\theoremstyle{definition}
\newtheorem{definition}[theorem]{Definition}
\newtheorem{remark}[theorem]{Remark}
\newtheorem{example}[theorem]{Example}

\newcommand{\mb}{\mathbf}
\newcommand{\fix}{\operatorname{Fix}}

\title{Determinants of simple theta curves and symmetric graphs}

\author{Matthew Elpers}
\author{Rayan Ibrahim}
\author{Allison H. Moore}

\address{Matthew Elpers\\ Department of Mathematics \& Applied Mathematics \\ 
  North Carolina State University \\Raleigh, NC 27606 \\ USA  
   }

\address{Rayan Ibrahim\\ Department of Mathematics \& Applied Mathematics \\ 
   Virginia Commonwealth University \\Richmond, VA 23284 \\ USA  
   }

\address{Allison H. Moore\\ Department of Mathematics \& Applied Mathematics \\ 
   Virginia Commonwealth University \\Richmond, VA 23284 \\ USA  
   }


\begin{document}

\begin{abstract} 

A theta curve is a spatial embedding of the $\theta$-graph in the three-sphere, taken up to ambient isotopy. 
We define the determinant of a theta curve as an integer-valued invariant arising from the first homology of its Klein cover. 
When a theta curve is simple, containing a constituent unknot, we prove that the determinant of the theta curve is the product of the determinants of the constituent knots. Our proofs are combinatorial, relying on Kirchhoff's Matrix Tree Theorem and spanning tree enumeration results for symmetric, signed, planar graphs.   

\end{abstract}
\maketitle

\section{Introduction}

A \emph{theta curve} $\vartheta$ is an embedding of the $\theta$-graph in the three-sphere, up to equivalence by ambient isotopy. The $\theta$-graph is the unique abstract graph consisting of two vertices connected by three parallel edges. Theta curves and other spatial graphs are generalizations of knots and links. In this article we study an integer-valued invariant of theta curves that we call the determinant $\det(\vartheta)$. Like the well-known determinant of links, this invariant can be defined as the order of the torsion subgroup of the first homology of a certain branched covering space (see Definition \ref{definition}). 
Every theta curve contains three \emph{constituent knots} $K_{ij}$, formed by taking pairs of edges $i, j \in \{a, b, c\}$. 
A \emph{simple} theta curve is one which contains at least one constituent knot that is unknotted \cite{Turaev}. 
For example, amongst prime theta curves of up to seven crossings, all 90 in the Litherland-Moriuchi table \cite{Moriuchi, Litherland} are simple (see Table \ref{det table}). 
The relationship between $\det(\vartheta)$ for simple theta curves and the determinants of its constituent knots is described by the following statement. 

\begin{restatable}[]{thm}{detprod}
\label{main}
Let $\vartheta$ be a simple theta curve with constituent knots $K_{ab}, K_{ac}, K_{bc}$. Then 
\[ \det(K_{ab})\cdot \det(K_{ac}) \cdot \det(K_{bc}) = \det(\vartheta). \] 
\end{restatable}

Our method for proving Theorem \ref{main} is combinatorial. 
By assumption one constituent, say $K_{ac}$, is an unknot with $\det(K_{ac})=1$. We relate the determinants of the other two constituent knots with counts of weighted spanning trees of Tait graphs that are derived from a diagram of $\vartheta$.  
We were surprised to find that the determinants of constituent knots of a theta curve provide a geometric interpretation of a purely graph-theoretic spanning tree enumeration formula. 
More specifically, Ciucu, Yan and Zhang applied the Matrix Tree Theorem to enumerate the spanning trees of a graph admitting an involutive symmetry via a product formula involving two smaller graphs \cite{CYZ, ZY}. In our context, we realize a graph that admits an involutive symmetry as the Tait graph of a strongly invertible knot (see Section \ref{sec: knots}) corresponding to the theta curve. The Goeritz matrix plays the role of the graph Laplacian, the determinant of which calculates the tree weight. We explicitly identify the two factors in the spanning tree enumeration formula with the determinants of the constituent knots of the theta curve, as realized by their Tait graphs.

\section{Knots and spatial theta curves}
\label{sec: knots}

We consider knots and theta curves to be smoothly embedded in the three-sphere, up to equivalence by ambient isotopy. Label edges of a theta curve by the letters $\{a, b, c\}$,  which may be thought of as non-identity elements of the Klein group $V\cong \mathbb{Z}_2\times \mathbb{Z}_2$. A theta curve is a special type of \emph{Klein graph} (see \cite{GilleRobert}), meaning a trivalent spatial graph endowed with a 3-edge coloring. Theta curves are also 3-Hamiltonian, meaning all of its $\{i,j\}$-colored subgraphs are connected. This means its constituents are knots, rather than links. 

Recall that the cyclic double cover $\widehat{X}_2$ of the complement of a knot, $X=S^3-N(K)$, is the regular covering space corresponding with the kernel of a homomorphism $\pi_1(X, x) \rightarrow H_1(X; \mathbb{Z}) \rightarrow \mathbb{Z}  \rightarrow \mathbb{Z}_2$. 
The branched double cover $\Sigma_2(S^3, K)$ may be obtained by gluing a solid torus to the boundary of $\widehat{X}$ via the map $(z_1, z_2)\mapsto (z_1, z_2^2)$  to extend the covering to a branched covering map $\Sigma_2(S^3, K) \rightarrow S^3$. 
It is a standard fact of knot theory that the branched double cover of a knot is a rational homology sphere and the determinant of a knot may be defined by $\det(K):=|H_1(\Sigma_2(S^3, K); \mathbb{Z})|$. 
See for reference \cite[Chapter 7]{Lickorish}. 

Given a theta curve $\vartheta\in S^3$ and complement $Y=S^3-N(\vartheta)$, we may similarly construct a covering space corresponding to the map $\pi_1(Y, x) \rightarrow  H_1(Y; \mathbb{Z}) \rightarrow \mathbb{Z}_2\times \mathbb{Z}_2$. This can be completed to a closed, oriented 3-manifold acted on by $V$ by gluing solid cylinders and cubes to the boundary in a procedure explicitly described by Gille and Robert in \cite[Proposition 2.6]{GilleRobert}. This manifold is the \emph{Klein cover} $\Sigma_\vartheta:=\Sigma(S^3, \vartheta)$ and has $\vartheta$ as the branching locus. 

\begin{definition}
\label{definition}
Let $\Sigma_\vartheta$ denote the Klein cover of a theta curve in $S^3$. The \emph{determinant} $\det(\vartheta)$ of $\vartheta$ is the order of $H_1(\Sigma_\vartheta ;\mathbb{Z})$. 
\end{definition}
One may visualize the Klein cover by iterating the branched double cover construction. 
One first constructs $\Sigma_2(S^3, K_{ac})$, branched over one of the constituent knots $K_{ac}=e_a\cup e_c$, then constructs a second branched covering of the manifold $\Sigma_2(S^3, K_{ac})$ branched over the knot $\tilde{e}_b$ that is the lift of the edge $e_b$. 
This will also yield $\Sigma_\vartheta \cong \Sigma_2(\Sigma_2(S^3, K_{ac}), \tilde{e}_b)$. 
The Klein cover of $\vartheta$ is unique, and so the order of $a, b, c$ in this procedure does not matter. 
A proof that the Klein cover of a spatial Klein graph in $S^3$ is unique up to diffeomorphism may be found in \cite[Proposition 2.8]{GilleRobert}. 
Consequently,  the determinant $\det(\vartheta)$ is a well-defined integer invariant of theta curves in the three-sphere.

\subsection{Simple thetas and strongly invertible knots}
\label{strongly invertible}

Consider the case that $\vartheta$ is simple. 
Up to relabeling, we may assume $K_{ab}$ is an unknot. 
Then $\vartheta$, together with this unknotted constituent, corresponds with a strongly invertible knot in the three-sphere as follows. 
The branched cover $\Sigma_2(S^3, K_{ac})$ is diffeomorphic to $S^3$, and the lift $\tilde{e}_b$ consists of two pre-images of $e_b$ joined at the two vertices of $\vartheta$ on the branching set. 
Recall that a knot $K$ in $S^3$ is \emph{strongly invertible} if there is an orientation-preserving involution  $h$ on $S^3$ such that $h(K)=K$ and $\fix(h)$ is a circle intersecting $K$ in two points \cite{Sakuma}. 
In our context, $\tilde{e}_b$ is strongly invertible. We write $\tilde{e}_b=(K, h)$ to emphasize the involution. 
For the reverse correspondence, let $(K, h)$ be any strongly invertible knot in the three-sphere. 
As a consequence of the Smith conjecture, $\fix(h)$ is unknotted and by definition, $(K, h)$ intersects $\fix(h)$ in two points. 
The quotient $K/h$ is an embedded closed arc. 
Thus $ \fix(h)\cup K/h= e_a\cup e_b  \cup e_c$ is a simple theta curve.

Recall that the branched double cover of any knot in $S^3$ is a rational homology sphere with first homology of odd order. Thus in the case $\vartheta$ is simple, $\det(\vartheta) = \det(K, h)$ is an odd integer. 
Note also that in the quotient under the action of the involution, a right-handed (respectively, left-handed) crossing in $(K, h)$ descends to a right-handed clasp in $\vartheta$, as in Figure \ref{incidence clasp}. We will make use of this observation later.

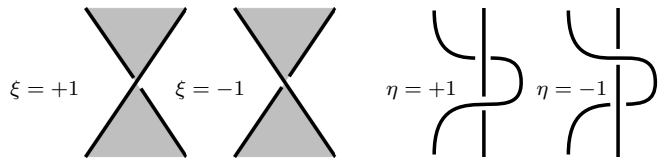
\begin{figure}[h!]	
\begin{tikzpicture}

\node[anchor=south west,inner sep=0] at (0,0) {\includegraphics[width=3in]{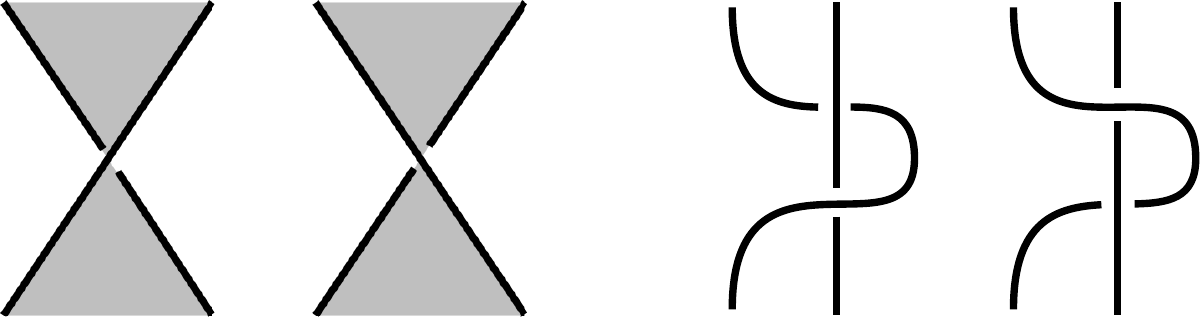}};

\footnotesize
\node[label=above right:{$\xi=+1$}] at (-1.2,0.6){};
\node[label=above right:{$\xi=-1$}] at (1,0.6){};
\node[label=above right:{$\eta=+1$}] at (3.8,0.6){};
\node[label=above right:{$\eta=-1$}] at (5.8,0.6){};
\end{tikzpicture}
\caption{Sign conventions for incidence numbers of crossings in a checkerboard shading (left) and for clasps in the quotient theta curve (right).  The shading indicates regions $X_i$ colored by an assignment $\varphi(X_i) = 0$. Both $\xi$ and $\eta$ are independent of strand orientation.}
\label{incidence clasp}
\end{figure}

\subsection{Goeritz matrices}
\label{subsec Goeritz}
The determinant of a knot or link can be calculated combinatorially as the determinant of an integral matrix associated to a knot diagram, due to a construction of Goeritz and Trotter \cite{Goeritz, Trotter}.  We review this following \cite[Chapter 9]{Lickorish} and apply it to simple theta curves below. 
Let $K$ be a knot with diagram $D_K$. Then $D_K$ admits two checkerboard colorings $\varphi$ of the regions $X=\{X_0,\dots, X_m\}$ of $D_K$, that is, there are two assignments $\varphi:X \to \{0,1\}$ where $\varphi(X_i) \ne \varphi(X_j)$ when $X_i$ and $X_j$ share a boundary curve. 

To each crossing $c$ of $D_K$ we associate a sign $\xi(c)$ with the convention in Figure \ref{incidence clasp}. 
Let $\{X_i \ | \ \varphi(X_i) = 0\} = \{B_0,\dots,B_n\}$ and let $C_{ij}$ be the set of crossings where $B_i$ and $B_j$ meet. 
We may then associate an $(n+1)\times (n+1)$ matrix $\widetilde{Q}_{D_K}$ to the diagram $D_K$ of $K$ with respect to the choice in shading $\varphi$. 
The matrix $\widetilde{Q}_{D_K}=[q_{ij}]$ is defined by
\[
	q_{ij} = \begin{cases}
	- \sum\limits_{c \in C_{ij}} \xi(c), & \text{if } i\neq j\\
	- \sum\limits_{k \neq i} q_{ik}, & \text{if } i= j.
	\end{cases}
\]
From this, $\det(K) = |\det(Q_{D_K})|$ where the \emph{Goeritz matrix} $Q_{D_K}$ is the $n\times n$ matrix obtained from $\widetilde{Q}_{D_K}$ by deleting any row and column. The result is independent of the choices in the knot diagram, the checkerboard coloring, labelling of regions, and the row and column selected for deletion. 

\section{Symmetric weighted graphs}
Let $G= (V(G), E(G))$ denote a graph and its vertex and edge sets. We will assume that graphs are undirected, but permit multi-edges, self-loops and edge weights. 
Recall that a \textit{spanning tree} $T\subseteq G$ is a connected acyclic subgraph with $V(T) = V(G)$. 
For a graph $G$ endowed with edge weights $\omega(e)$, define $\omega(G) \coloneqq \prod_{e\in E(G)} \omega(e)$. 
Define the \textit{tree weight} of $G$ by
\begin{equation}
\label{tree weight}
\tau(G) = \sum_{T \subseteq G} \omega(T)= \sum_{T \subseteq G} \prod_{e\in E(T)} \omega(e),
\end{equation}
where the sum is over spanning trees of $G$. For graphs with edge weights all equal one, $\tau(G)$ is simply the number of spanning trees of $G$. 
We restrict our attention to edge weights $\omega(e)$ in the multiplicative group $\{-1, +1\}$, and so tree weights will take on integer values. There is a well known method to count spanning trees, or more generally the tree weight, using the graph Laplacian. 
 
\begin{definition}
Let $G$ be a weighted graph with $V(G) = \{v_0,\dots,v_n\}$. An $(n+1)\times (n+1)$ matrix called the \textit{Laplacian} of $G$, $\widetilde{L}_G$ = $[\ell_{ij}]$, is defined by
\begin{equation}
\label{laplacian}
\ell_{ij} =
\begin{cases}
-\omega_{ij} &\text{ if } i \ne j  \\ 
\sum_{k\ne i} \omega(e_{ik}) &\text{ if } i = j. 
\end{cases}
\end{equation}
Here, $\omega_{ij}$ is the sum of edge weights over all edges connecting $v_i$ and $v_j$.
\end{definition}
The following theorem is often attributed to Kirchhoff, and different versions are due to Bott-Mayberry \cite{BottMayberry} and Tutte \cite{Tutte}.
See also the exposition in \cite{Biggs}.

\begin{theorem}[Matrix Tree Theorem]\label{thm:Kirchhoff}
Let $G$ be a graph and let $L$ be the reduced Laplacian of $G$, obtained by deleting any row and column from $\tilde{L}$. Then $|\det(L)| = \tau(G)$.
\end{theorem}

\subsection{Tait graphs}
The procedure for calculating the determinant from a Goeritz matrix yields an equivalent graph theoretic method using the \emph{Tait graph} of a knot diagram. Let $\varphi$ be a checkerboard coloring of a diagram $D_K$ of $K$. The fact that there exists a checkerboard coloring $\varphi$ for any $D_K$ can be proven in the following way: By forgetting crossing information, a knot diagram yields a planar four-valent graph $G$. Observe that the dual $G^\perp$ cannot contain any odd cycles, otherwise $G$ would contain a vertex of odd degree. Therefore $G^\perp$ is bipartite, so the faces of $G$ are two-colorable. 

Thus, the coloring determines a pair of planar dual graphs $G$ and $G^\perp$. The vertices $V(G)$ correspond with the shaded regions  $B = \{X_i \ | \ \varphi(X_i) = 0\}$ and the vertices $V(G^\perp)$ with unshaded regions  $W = \{X_i \ | \ \varphi(X_i) = 1\}$. Edges in both graphs correspond to incidences between regions at crossings, with edge weight $\omega(e)=\omega(e^\perp) = \xi(c)$. Examples of Tait graphs are shown in Figure \ref{example-948}. 

Combining the Matrix Tree Theorem and the Goeritz matrix formulation of the determinant of a knot, we have
\[
	\det(K) = |\det(Q_{D_K})|  = |\det(L_G)| = \tau(G)
\]
where $Q_{D_K}$ is the reduced Goeritz matrix of $K$ corresponding with any diagram $D_K$ of $K$, $L_G$ is the reduced Laplacian of the corresponding Tait graph $G$, and where $\tau(G)$ is tree weight of $G$. 

We now collect several lemmas pertaining to signed graphs that we will need in later sections. 

It is a standard result in graph theory that the number of spanning trees of a planar graph $G$ is equal to the number of spanning trees of $G^\perp$. For graphs with arbitrary edge weights, this statement is false; for a counterexample, take a triangle with edge weights 1, 2, 3. For edge weights in the multiplicative group  $\{-1, +1\}$, though, the statement generalizes as follows. 

\begin{lemma}
\label{dual tree weight}
For planar graphs with edge weights in $\{-1, +1\}$, $|\tau(G)| = |\tau(G^\perp)|$. 
\end{lemma}

\begin{proof}
Let $G$ be a planar graph and $G^\perp$ its dual. 
There is a bijection $E(G) \to E(G^\perp)$ which sends $e \in E(G)$ to $e^\perp \in E(G^\perp)$. 
In particular, every edge $e$ borders faces $F_1$ and $F_2$ and $e^\perp = (F_1,F_2) \in E(G^\perp)$ is the edge uniquely corresponding to $e$. 
As a result, there is a bijection $f$ that associates each spanning tree $T\subseteq G$ with a spanning tree $f(T) \subseteq G^\perp$, where 
	\[ f(T) = G^\perp - \{e^\perp : e \in T\} = (G - T)^\perp. \] 
See for example \cite{Lovasz}. 
We extend $f$ to weighted trees by assigning dual edges the same weight, i.e., $\omega(e) = \omega(e^\perp)$.

By assumption $\omega(e) \in \{-1,1\}$ for all $e \in E(G)$. Recall that $\omega(G) \coloneqq \prod_{e\in E(G)} \omega(e)$. 
From $f$ we can deduce that for any spanning tree $T \subseteq G$,
\[ \omega(G) = \tau(T)\tau(f(T)) \]
and there are two cases.

\begin{enumerate}[leftmargin=0.5cm]
\item[] \textbf{Case 1.} $\omega(G) = 1$. Then it must be the case that for all spanning trees $T$, $\tau(T) = \tau(f(T)) = 1$ or $\tau(T) = \tau(f(T)) = -1$. That is, for all spanning trees $T$, $\tau(T) = \tau(f(T))$.

\item[] \textbf{Case 2.} $\omega(G) = -1$. Then it must be the case that for all spanning trees $T$, $\tau(T) =1$ and $\tau(f(T)) = -1$ or $\tau(T) =-1$ and $\tau(f(T)) = 1$. That is, for all spanning trees $T$, $\tau(T) = - \tau(f(T))$.
\end{enumerate}

So either $\tau(G) = \tau(G^\perp)$ or $\tau(G) = -\tau(G^\perp)$.
\end{proof}

\begin{lemma}
\label{multiedge}
	Let $G$ be a graph containing an edge $e = (v, w)$ of weight $\omega(e)$. Let $G'$ be $G-\{e\} \cup \{ e_1, e_2\}$ where $e_1 =(v, w)= e_2$ are edges of weight $\frac{1}{2}\omega(e)$. Then $\tau(G) = \tau(G')$. 
\end{lemma}
\begin{proof}
For every spanning tree $T$ in $G$ that contains $e$ there exist exactly two spanning trees in $G'$, each of tree weight $\frac{1}{2}\omega(T)$.  
\end{proof}

A version of the following lemma is proved in \cite[Lemma 6]{CYZ} for unweighted graphs.  
Here, we are interested in counting spanning trees where the graphs inherit edge weights from the crossings of knot diagrams, and edge subdivisions will occur in the Tait graphs of our constituent knots. 
Hence, we extend their lemma to the specific case of graphs with edge weights $\omega(e)\in\{-1, +1\}$.

\begin{lemma}
\label{subdivide}
	Let $G_0$ be a graph with vertices $V_0$ and edges $E_0$. 
	Let $a$, $b$ and $x$ be three vertices distinct from $V_0$. 
	Construct a graph $G = (V, E)$ by taking $V = V_0 \cup \{a,b\}$, and letting $E = E_0\cup (a, b)\cup S$, where $S$ is any set of edges of the form $(v, a)$ or $(v, b)$, where  $v\in V_0$; specify the edge weight of $\omega(a,b)=\pm 1/2$. 
	Construct a graph $G' = (V', E')$ by taking $V' = V_0 \cup \{a,x,b\} $, and letting $E' = E_0\cup (a,x) \cup (b,x) \cup S$, where $S$ is as in $G$; specify the edge weights $\omega(a,x)=\omega(b,x)=\pm 1$, in agreement with the sign of $\omega(a,b)$. Then
	\[
		|2\tau(G)| = |\tau(G')|.
	\]
	\end{lemma}
	
	\begin{proof}
	As in the proof of  \cite[Lemma 6]{CYZ}, we may partition the spanning trees of $G$ into two sets $C_1\cup C_2$, where spanning trees in $C_1$ contain edge $(a, b)$ and spanning trees in $C_2$ do not contain edge $(a, b)$. Likewise, partition the spanning trees of $G'$ into three sets
$C1' \cup C2' \cup C3'$, where spanning trees in $C_1'$ contain both $(a, x), (b, x)$, where trees in $C_2'$ contain $(a, x)$ but not $(b, x)$, and trees in $C_3'$ contain $(b, x)$ but not $(a, x)$.

	There exists a bijection $f:C_1'\rightarrow C_1$ obtained by deleting the vertex $x$ and adding the edge $(a,b)$. 
	If the weights of the edges $(a,b)\in G$ and $(a,x), (b,x)$ in $G'$ are all positive, then the bijection satisfies $\omega(f(T')) = \frac{1}{2}\omega(T')$ for all spanning trees $T'$ in $C_1'$, whereas if the edge weights are all negative, then $\omega(f(T')) = -\frac{1}{2}\omega(T')$. 
	There are also bijections $g:C_2'\rightarrow C_2$ and $h:C_3' \rightarrow C_2$, obtained by contracting the edge $(a, x)$ or $(b, x)$, respectively. 
	In this case, when the edge weights are all positive, then $\omega(g(T')) = \omega(T')$ and $\omega(h(T')) = \omega(T')$, whereas if  the edge weights are all negative, then $\omega(g(T')) = \omega(T')$ and $\omega(h(T')) = \omega(T')$.  

Finally, observe that
\[
	\tau(G') = \sum_{T'\in C_1'} \omega(T') + \sum_{T'\in C_2'}\omega(T') + \sum_{T'\in C_3'}\omega(T')  =\pm2  \sum_{T\in C_1}\omega(T) \pm 2\sum_{T\in C_2}\omega(T) =\pm 2\tau(G),
\]	
where the sign in front of the summation is positive/negative when the edge weights $(a,b)\in G$ and $(a,x), (b,x)\in G'$ are all positive/negative, respectively.
	\end{proof}

\begin{remark}
\label{selfloop}
Non-simple graphs containing self-loops or multiedges may result from Tait graphs of knot diagrams. 
For edges that are self-loops, $\tau(G) = \tau(G-e)$. 
\end{remark}
	
\subsection{Spanning trees of graphs with involutive symmetry}

We will now show how the relationship between the determinants of the constituent knots and theta curve is described by counting spanning trees of graphs with involutive symmetry. Here, $G = (V(G), E(G))$ is the weighted Tait graph of $(K, h)$, with symmetry from the involution $h$. 
The following algorithms constructing the graphs $G_L$ and $G_R$ are due to Zhang-Yan \cite[Theorem 2.1]{ZY}, generalizing unweighted versions due to Ciucu-Yan-Zhang in \cite[Theorem 4]{CYZ}. 
The involution $h$ partitions $V(G)$ into three sets: $V_L\cup V_C\cup V_R$, where $V_L =\{ v_1, \cdots, v_n\}$ consists of vertices on the left side of the axis of involution, $V_R =\{ v_1', \cdots, v_n'\}$ are vertices on the right, and $V_C =\{ w_1, \cdots, w_m\}$ are vertices lying on the axis.

\begin{definition}($G_R$ and $G_L$ \cite{ZY}.)
\label{defn}
Two weighted graphs $G_L$ and $G_R$ are obtained from $G$ as follows. To form $G_R$:
\begin{enumerate}[label=(\roman*)]
	\item Take the subgraph of $G$ induced by $V_R\cup V_C$.
	\item For every edge $e=(w_i, w_j)$ along the axis of involution, reduce the weight by half.

\end{enumerate}
To form $G_L$:
\begin{enumerate}[label=(\roman*)]
	\item Take the subgraph of $G$ induced by $V_L$ together with a new vertex $u$. 
	\item For each edge $e=(v_i, v_i')$ with weight $\omega(e)$, add an edge $(u,v_i)$ with weight $2\omega(e)$. 
	\item For each edge $e=(v_i,w_j)$ add an edge $(u,v_i)$ with weight $\omega(e)$.
	\end{enumerate}
\end{definition}	
	
With these defined,	
\begin{theorem}(Zhang-Yan \cite[Theorem 2.1]{ZY})
\label{zy theorem}
Suppose that $G = (V(G), E(G))$ is a weighted graph with an involution $h$ and that $G_L, G_R$, and $V_C$ are defined as above. Then the tree weight of $G$ is given by
\[	
	\tau(G) = 2^{m-1}\tau(G_L )\tau(G_R)
\]
where $m$ is the number of vertices of $V_C$.
\end{theorem}

\begin{remark}
Because $(K, h)$ is strongly invertible, the edge set $E(G)$ of the Tait graph $G$ contains edges of the form $e=(v_i, v_j')$ only if $i=j$. Thus we have omitted items ($3$) and ($2'$) from the definition appearing in \cite{ZY}. 
\end{remark}

\section{Proof of Theorem \ref{main}}

Assume that $K_{ac}=e_a\cup e_c$ is unknotted, and call $K_{ab} = e_a\cup e_b$ and $K_{bc}=e_b\cup e_c$.

As described in Section \ref{strongly invertible}, $K_{ac}$ can be viewed as the fixed set $\fix(h)$ of the involution for some strongly invertible knot $(K, h) = \tilde{e}_b$. 
We may assume that any diagram of $K$ is symmetric with respect to $h$ and view $K_{ac}$ as a vertical axis $\alpha$ with the point at infinity. 
By definition, $(K, h)$ intersects $\alpha$ in exactly two points. 
This partitions $\alpha \cup \{\infty\} =e_a \cup e_c$ into `two rooms' along which the diagram admits
a uniform checkerboard coloring pattern from wall to wall, as in Figure \ref{clasp shading}. 
More precisely, symmetry implies that given any edge $e\in E(G)$ or $e\in E(G^\perp)$, either $\alpha$ intersects $e$ in exactly one point, $\alpha$ and $e$ are disjoint, or  $\alpha$ intersects $e$ in $e$. 
Recall that there are two choices of a checkerboard shading $\varphi$ of diagram of $(K, h)$. 
For exactly one choice of shading $\varphi$, the following holds for all edges:
\begin{eqnarray}
\label{parallel perpendicular}
\begin{aligned}
&\text{In } e_a: \quad e \cap e_a = e  \text{ and } e^\perp \cap e_a = 1\text{ point},\\
&\text{In } e_c: \quad e \cap e_c = 1 \text{ point and } e^\perp \cap e_c = e^\perp. 
\end{aligned}
\end{eqnarray}
The other choice in shading will yield an equivalent statement interchanging $e$ and $e^\perp$. As a consequence of \eqref{parallel perpendicular}, we have:

\begin{figure}	
	\includegraphics[width=3.4in]{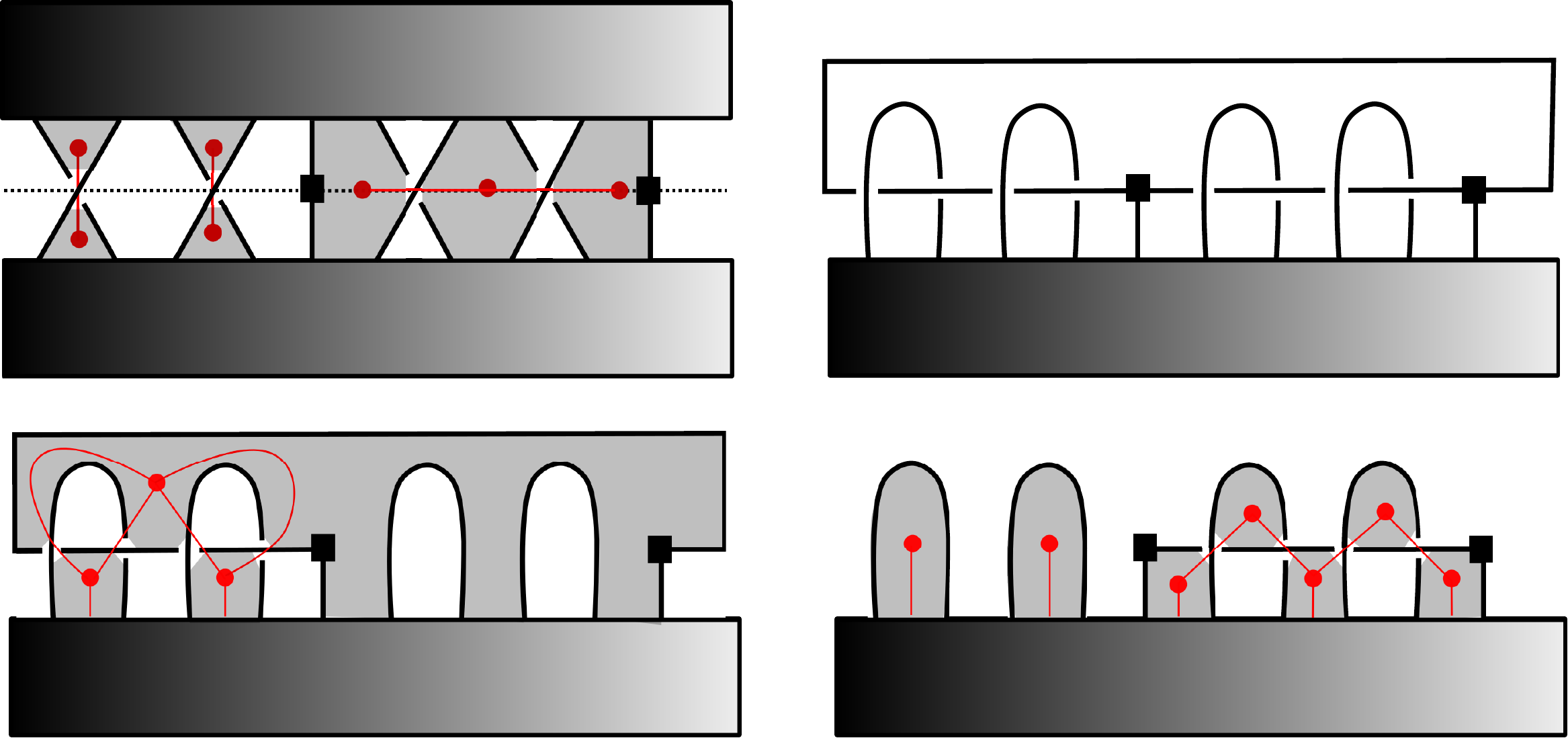}
	\caption{(Top left) Checkerboard shading on $(K, h)$ showing `two rooms' of shading patterns along the axis. (Top right) $\vartheta = K/h \cup e_a\cup e_c$. (Bottom) The shading of $(K, h)$ induces a shading on the constituent knots $K_{bc}$ (left) and $K_{ab}$ (right) of $\vartheta$. Parts of the Tait graphs are indicated in red.} 
	\label{clasp shading}
\end{figure}

\begin{lemma}
\label{induce shading}
A checkerboard shading of $(K, h)$ in $S^3$ induces a checkerboard shading on the constituent knots $K_{ab}$ and $K_{bc}$ in the quotient diagram of $\vartheta = K/h \cup \alpha$ in $S^3$. 
\end{lemma}
\begin{proof}
Choose a checkerboard shading in a symmetric diagram of $(K, h)$. Quotient via the involution $h$ to obtain a theta curve $\vartheta = K/h \cup e_a \cup e_c$. In $K_{ab} = K/h \cup e_a$, the checkerboard shading at the crossings of $(K, h)$ along $e_a$ descend in the quotient to checkerboard shaded clasps, and the shaded crossings along $e_c$ descend to shaded `fingers.' See Figure \ref{clasp shading}. Similarly in $K_{bc}=K/h\cup e_c$,  the checkerboard shading at crossings along $e_c$ descends to shaded clasps, and the shaded crossings along $e_a$ descend to shaded fingers. Away from the axis, the checkerboard shadings in the diagram of $K_{ab}$ and $K_{bc}$ agree with that of $(K, h)$. 
\end{proof}

Let $G, G_{ab}, G_{bc}$ denote the Tait graphs for $(K,h), K_{ab}, K_{bc}$, respectively.

\begin{lemma} 
\label{factors}
For one choice of checkerboard shading of $(K, h)$, we have
$2^{m-1}\tau(G_R) = \tau(G_{ab})$ 
and $\tau(G_L) = \tau(G_{bc})$. With other choice, 
$2^{m-1}\tau(G^\perp_R) = \tau(G\perp_{bc})$ and  
$\tau(G^\perp_L) = \tau(G\perp_{ab})$. 
\end{lemma}

In the second case, Lemma \ref{dual tree weight} implies $|\tau(G)| = |\tau(G^\perp)|$. 

\begin{proof}[Proof of Lemma \ref{factors}]
Fix a symmetric diagram of $(K, h)$ and by convention, let $e_c$ be the unbounded arc of the axis. 
The Tait graph $G$ of the diagram is symmetric, planar and connected.
The edge weights $\omega(e)\in\{-1, +1\}$ for $e\in G$ are induced from the incidence numbers $\xi(c)$ at the crossings, where the signs of the $\xi(c)$ depend on the choice in the checkerboard shading of the diagram. 
Specify the shading $\varphi$ of $(K, h)$ 
so that \eqref{parallel perpendicular} holds. 
With this choice, the unbounded region of the diagram is unshaded. 
By Lemma \ref{induce shading}, $\varphi$ induces a shading on the constituent knot $K_{ab} = e_b\cup e_a$ corresponding with Tait graph $G_{ab}$.  
Under the action of the involution, edges in $G$ that are disjoint from $\alpha$ map bijectively to edges in $G_{ab}$. Edges intersecting $\alpha$ in a point do not map to edges in $G_{ab}$, and edges that lie along $\alpha$ map to a subdivided edge in $G_{ab}$. 
In particular, to form $G_{ab}$: 
\begin{enumerate}[label=(\roman*)]
	\item Take the subgraph of $G$ induced by $V_R\cup V_C$.
	\item For every edge $e=(w_i, w_j)$ along the axis of involution, subdivide $e$ into $(w_i, x)\cup (x,w_j)$ and set $\omega(w_i, x) =\omega (x,w_j) = \frac{1}{2}\omega(e)$. 
\end{enumerate}
This nearly agrees with the definition of $G_R$; Lemma \ref{subdivide} then implies $2^{m-1}\tau(G_R) = \tau(G_{ab})$. 

Consider now $K_{bc}$ with Tait graph $G_{bc}$. 
Under the action of the involution, edges in $G$ that are disjoint from $\alpha$ map bijectively to edges in $G_{bc}$. 
Edges that lie along $\alpha$ do not map to edges in $G_{bc}$. Edges $e = (v_i, v_i')$ in $G$ that intersect $\alpha$ in a point map to a pair of edges in $G_{bc}$. (This pair of edges is dual to a subdivided edge in $G_{bc}^\perp$; see Figure \ref{clasp shading}.) 
In particular, to form $G_{bc}$: 
\begin{enumerate}[label=(\roman*)]
	\item Take the subgraph of $G$ induced by $V_L$ together with a new vertex $u$. 
	\item For each edge $e=(v_i, v_i')$ with weight $\omega(e)$, add a pair of edges edges $e_1=(u,v_i)=e_2$ each with weight $\omega(e)$. 
	\item For each edge $e=(v_i,w_j)$ add an edge $(u,v_i)$ with weight $\omega(e)$.
	\end{enumerate}
This nearly agrees with the definition of $G_L$; the difference is the factor of $2$ in the edge weight in item (ii), which here manifests as a pair of edges. Thus $\tau(G_L) = \tau(G_{bc})$. 

Finally, let us consider the other choice in shading. 
Equation \ref{parallel perpendicular} becomes an equivalent statement with $e$ and $e^\perp$ interchanged. 
Duality preserves connectedness, planarity, symmetry and edge weights. 
The above argument applies, \emph{mutatis mutandis}: interchange $G^\perp$ and $G$, and interchange $K_{ab}$ and $K_{bc}$. 
\end{proof}

We can now prove the main result:
\detprod*

\begin{proof}

By Lemma \ref{dual tree weight}, the determinants of $(K,h), K_{ab}$, and $K_{bc}$ may be calculated by the tree weights of $G, G_{ab}$, and $G_{bc}$, respectively (or equivalently by the tree weights of $G^\perp, G^\perp_{ab}$, and $G^\perp_{bc}$).  Hence, by Theorem \ref{zy theorem}, we have
\[	
	\det(\vartheta)=	 \tau(G) = 2^{m-1}\tau(G_L )\tau(G_R) = \tau(G_{ab}) \tau(G_{bc}) = \det(K_{ab})\det(K_{ac})\det(K_{bc}). \qedhere
\]
\end{proof}


\subsection{Examples}

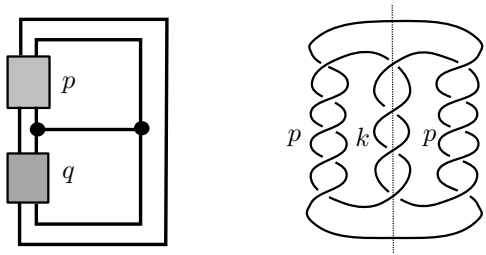
\begin{figure}	
\begin{tikzpicture}

\node[anchor=south west,inner sep=0] at (0,0) {\includegraphics[width=2.5in]{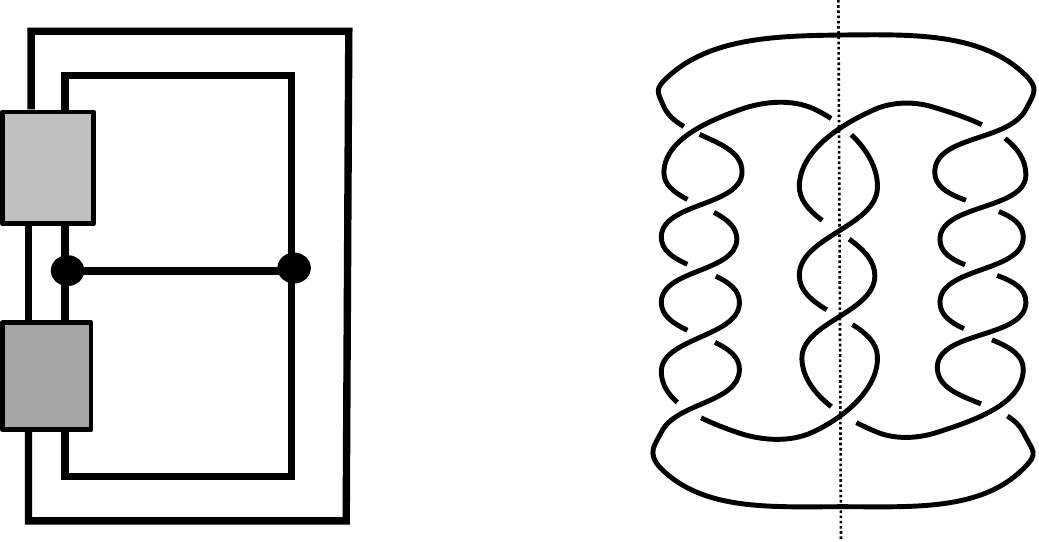}};

\node[label=above right:{$q$}] at (0.5,0.7){};
\node[label=above right:{$p$}] at (0.5,1.9){};
\node[label=above right:{$p$}] at (3.5,1.2){};
\node[label=above right:{$k$}] at (4.4, 1.2){};
\node[label=above right:{$p$}] at (5.3, 1.2){};
\end{tikzpicture}
\caption{The $\vartheta$-curve $\vartheta(p,q)$ with constituent knots $T(p+q, 2)$, $T(p, 2)$, and $U$. The $P(p,q,p)$ pretzel knot is the corresponding strongly invertible knot.}
\label{pretzel}
\end{figure}

\begin{example}
Let $\vartheta(p, q)$, with $p$ odd and $q=2k$ even, be the $\vartheta$-curve pictured in Figure \ref{pretzel}. 
The three constituent knots are the unknot, and the torus knots $T(p+q, 2)$ and $T(p, 2)$. 
By Theorem \ref{main}, $\det(\vartheta) = \det( T(p+q, 2)) \cdot \det ( T(p, 2)) = (p+q)\cdot p = p^2 + pq$. 
The pretzel knot $P(p,k,p)$ is the strongly invertible knot that corresponds with $\vartheta(p,q)$, and it also has determinant $p^2+pq$. 
\end{example}

\begin{example}
Consider the strongly invertible knot $(K, h) = 9_{48}$, pictured with an axis of involution in Figure \ref{example-948}. Its quotient under the involution, together with the axis, forms a spatial theta curve whose diagram contains 9 crossings. The two constituent knots  $K_{ab}= 3_1$ and $K_{bc}= 6_1$ are shown in the figure. The determinants of $3_1, 6_1, 9_{48}$ are $3, 9, 27$. The Tait graphs $G, G^\perp$ for $(K, h)$ and $G_{ij}, G^\perp_{ij}$ for $K_{ij}$ are also illustrated. 
\end{example}

\begin{figure}	
	\includegraphics[width=4in]{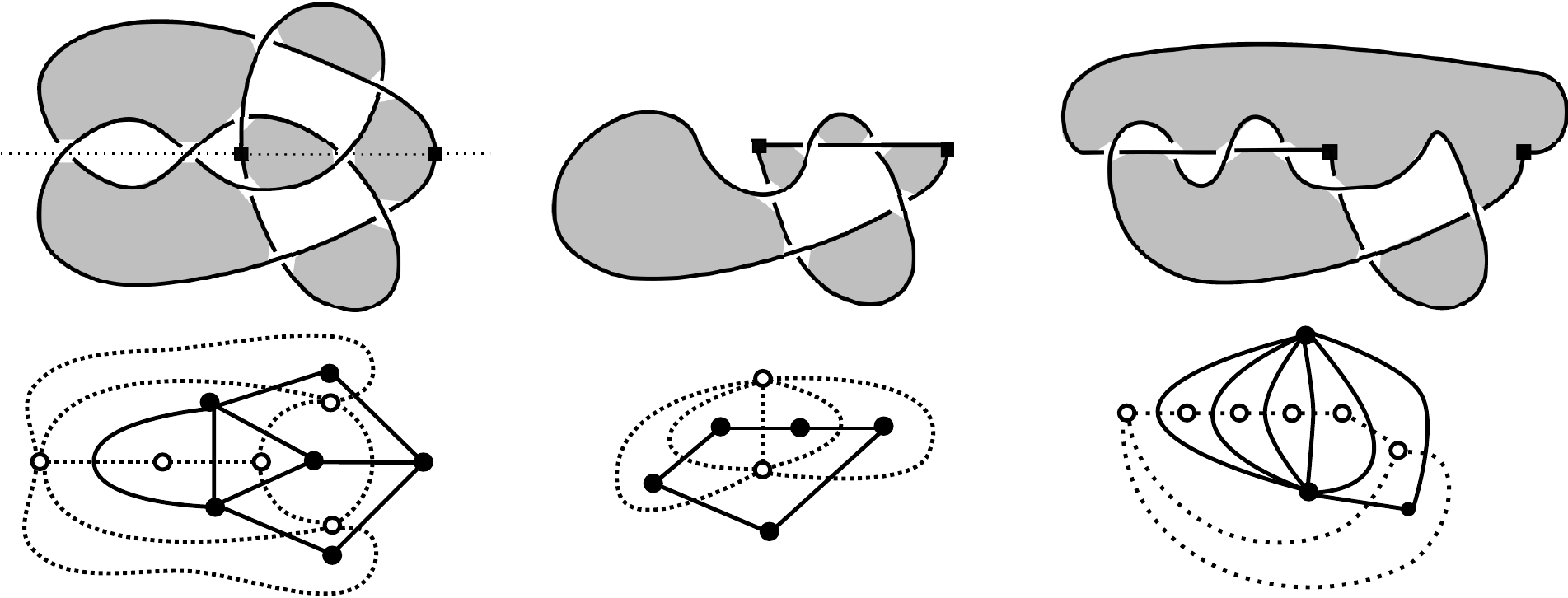}
	\caption{(Left) Strongly invertible knot $9_{48}$ with axis of involution. (Center) Constituent knot $K_{ab}$ is the trefoil $3_1$. (Right) Constituent knot $K_{bc}$ is the knot $6_1$. Corresponding Tait graphs and their duals are shown below.}
	\label{example-948}
\end{figure}

\begin{example}
Constituent knots for all theta curves in the Litherland-Moriuchi table were previously determined by Baker, Buck and O'Donnol in \cite[Table 2]{BBO}. We can now apply Theorem \ref{main} to compute the values of $\det(\vartheta)$ for all of the theta curves tabulated. An augmented table including the values of $\det(\vartheta)$ and $\det(K)$ for constituent knots is displayed as Table \ref{det table} below.
\end{example}

\section{Discussion}

Our definition of $\det(\vartheta)$ implicitly relies on the fact that a theta curve is a spatial trivalent planar graph admitting a 3-edge coloring, also called a \emph{Tait coloring}. Such Tait colorings play an important role in both the history of the four color theorem \cite{Tait} and in current gauge theoretic strategies for a new proof \cite{KM1}. 
In a completely different direction, the study of spatial theta curves is experiencing growing relevance in mathematical modeling of DNA replication and protein entanglement \cite{OSB, Sulkowska, DTGSS}.  This motivates a general effort to tabulate and differentiate theta curves \cite{Moriuchi}, and to further develop invariants and properties of these objects as in \cite{Wolcott, Litherland, Yamada, Kauffman:graphs, Kauffman:theta, BBO, BBMOT} (to name just a few). 

In this article, we present one possible definition of $\det(\vartheta)$. It is computed combinatorially, but is essentially a measure of the homology of the Klein cover. Alternatively, given any Alexander polynomial-type invariant, one could evaluate it at $t=-1$ to obtain a (possibly equivalent) definition. A generalization of the Alexander module was given by Litherland in \cite{Litherland}. An Alexander polynomial invariant for MOY-graphs was defined by Bao-Wu \cite{BaoWu}, but its specialization at $t=-1$ apparently counts the number of spanning trees of the underlying abstract graph. 
\begin{question}Two questions we pose:
\begin{itemize}
\item Do other definitions of the determinant that come from other strategies agree with the definition presented here?
\item Can a spanning tree enumeration strategy can be used to calculate $\det(\vartheta)$ for non-simple theta curves or other Klein graphs?
\end{itemize}
\end{question}

\subsection*{Acknowledgements}
We thank Ken Baker, Dorothy Buck and Danielle O'Donnol for providing the \LaTeX~for Table \ref{det table}, and Glenn Hurlbert for helpful discussions. ME, RI, and AHM were partially supported at VCU by The Thomas F. and Kate Miller Jeffress Memorial Trust, Bank of America, Trustee. RI and AHM were partially supported by National Science Foundation DMS-2204148. 
 
\begin{table}[htp]
  \caption{Theta curves through 7 crossings, their constituents, and their determinants. All of the constituent knots in the Litherland-Moriuchi table are simple. }
  \ra{1.2}
  \scriptsize
\begin{center}
\begin{multicols}{3}
\begin{tabular}{@{}lccccc@{}}\toprule
    $\vartheta$  &\multicolumn{2}{c}{C. Knots} & \multicolumn{2}{r}{det($K_{ij}$)} & det($\vartheta$) \\
    \midrule
    $\mb{3_1}$ & 2x$0_1$ & $3_1$ & &1,3& 3\\ 
    $\mb{4_1}$ & 2x$0_1$ & $4_1$ & &1,5& 5\\ 
    $\mb{5_1}$ & 3x$0_1$ &  & &1& 1\\ 
    $\mb{5_2}$ & 2x$0_1$ & $3_1$ & &1,3& 3\\ 
    $\mb{5_3}$ & 2x$0_1$ & $5_1$ & &1,5& 5\\ 
    $\mb{5_4}$ & $0_1$ & $3_1$ & $5_1$ &3,5& 15 \\ 
    $\mb{5_5}$ & 2x$0_1$ & $5_2$ & &1,7& 7\\ 
    $\mb{5_6}$ & 2x$0_1$ & $5_2$ & &1,7& 7\\ 
    $\mb{5_7}$ & $0_1$ & $3_1$ & $5_2$&3,7& 21\\ 
    $\mb{6_1}$ & 3x$0_1$ &  & &1& 1\\ 
    $\mb{6_2}$ & 2x$0_1$ & $3_1$ & &1,3& 3\\ 
    $\mb{6_3}$ & $0_1$ & $3_1$ & $4_1$ &3,5& 15\\ 
    $\mb{6_4}$ & $0_1$ & $3_1$ & $4_1$ &3,5& 15\\ 
    $\mb{6_5}$ & 2x$0_1$ & $6_1$ & &1,9& 9\\ 
     $\mb{6_6}$ & 2x$0_1$ & $6_1$ & &1,9& 9\\ 
     $\mb{6_7}$ & 2x$0_1$ & $6_1$ & &1,9 & 9 \\ 
     $\mb{6_8}$ & $0_1$ & $4_1$ & $6_1$ &5,9 & 45 \\ 
     $\mb{6_9}$ & 2x$0_1$ & $6_2$ & &1,11 &11 \\ 
     $\mb{6_{10}}$ & 2x$0_1$ & $6_2$ & &1,11& 11\\
     $\mb{6_{11}}$ & 2x$0_1$ & $6_2$ & &1,11& 11\\ 
      $\mb{6_{12}}$ & $0_1$ & $3_1$ & $6_2$ &3,11& 33\\ 
      $\mb{6_{13}}$ & $0_1$ & $4_1$ & $6_2$ &5,11& 55\\ 
     $\mb{6_{14}}$ & 2x$0_1$ & $6_3$ & &1,13& 13\\ 
     $\mb{6_{15}}$ & 2x$0_1$ & $6_3$ & &1,13& 13\\ 
     $\mb{6_{16}}$ & $0_1$ & $3_1$ & $6_3$ &3,13& 39\\ 
     $\mb{7_1}$ & 3x$0_1$ &  & &1& 1\\ 
    $\mb{7_2}$ & 3x$0_1$ &  & &1& 1\\ 
    $\mb{7_3}$ & 3x$0_1$ &  & &1& 1\\ 
    $\mb{7_4}$ & 3x$0_1$ &  & &1& 1\\ 
    $\mb{7_5}$ & 2x$0_1$ & $3_1$ & &1,3& 3\\ 
   \bottomrule
      \end{tabular}

\begin{tabular}{@{}lccccc@{}}\toprule
    $\vartheta$  &\multicolumn{2}{c}{C. Knots} & \multicolumn{2}{r}{det($K_{ij}$)} & det($\vartheta$) \\
    \midrule
    
     $\mb{7_6}$ & 2x$0_1$ & $3_1$ & &1,3& 3\\ 
    $\mb{7_7}$ & 2x$0_1$ & $3_1$ & &1,3& 3\\ 
    $\mb{7_8}$ & $0_1$ & $3_1$ &$3_1$ &3,3& 9\\ 
    $\mb{7_9}$ & $0_1$ & $3_1$ &$3_1$ &3,3& 9\\ 
    $\mb{7_{10}}$ & $0_1$ & $3_1$ &$3_1$ &3,3& 9\\ 
    $\mb{7_{11}}$ & 2x$0_1$ & $5_2$ & &1,7& 7\\ 
    $\mb{7_{12}}$ & 2x$0_1$ & $4_1$ & &1,5& 5\\ 
    $\mb{7_{13}}$ & 2x$0_1$ & $4_1$ & &1,5& 5\\ 
    $\mb{7_{14}}$ & $0_1$ & $4_1$ &$4_1$ &5,5& 25\\ 
    $\mb{7_{15}}$ & 2x$0_1$ & $5_1$ & &1,5& 5\\ 
       $\mb{7_{16}}$ & 2x$0_1$ & $5_1$ & &1,5& 5\\ 
    $\mb{7_{17}}$ & 2x$0_1$ & $5_1$ & &1,5& 5\\ 
    $\mb{7_{18}}$ & $0_1$ & $5_1$ & $5_2$&5,7& 35\\ 
    $\mb{7_{19}}$ & 2x$0_1$ & $5_2$ & &1,7& 7\\ 
    $\mb{7_{20}}$ & 2x$0_1$ & $5_2$ & &1,7& 7\\ 
    $\mb{7_{21}}$ & 2x$0_1$ & $5_2$ & &1,7& 7\\ 
    $\mb{7_{22}}$ & $0_1$ & $3_1$ & $5_2$&3,7& 21\\ 
    $\mb{7_{23}}$ & $0_1$ & $4_1$ & $5_2$&5,7& 35\\ 
    $\mb{7_{24}}$ & $0_1$ & $4_1$ & $5_2$&5,7& 35\\ 
    $\mb{7_{25}}$ & 2x$0_1$ & $7_1$ & &1,7& 7\\ 
    $\mb{7_{26}}$ & $0_1$ & $3_1$ & $7_1$&3,7& 21\\ 
    $\mb{7_{27}}$ & $0_1$ & $5_1$ & $7_1$&5,7& 35\\ 
    $\mb{7_{28}}$ & 2x$0_1$ & $7_2$ & &1,11& 11\\ 
    $\mb{7_{29}}$ & 2x$0_1$ & $7_2$ & &1,11& 11\\ 
    $\mb{7_{30}}$ & 2x$0_1$ & $7_2$ & &1,11& 11\\ 
    $\mb{7_{31}}$ & $0_1$ & $3_1$ & $7_2$&3,11& 33\\ 
    $\mb{7_{32}}$ & $0_1$ & $5_2$ & $7_2$ &7,11& 77\\ 
    $\mb{7_{33}}$ & 2x$0_1$ & $7_3$ & &1,13& 13\\ 
    $\mb{7_{34}}$ & 2x$0_1$ & $7_3$ & &1,13& 13\\ 
    $\mb{7_{35}}$ & $0_1$ & $3_1$ & $7_3$&3,13& 39\\    \bottomrule
          \end{tabular}

\begin{tabular}{@{}lccccc@{}}\toprule
    $\vartheta$  &\multicolumn{2}{c}{C. Knots} & \multicolumn{2}{r}{det($K_{ij}$)} & det($\vartheta$) \\
    \midrule
    
    $\mb{7_{36}}$ & $0_1$ & $5_1$ & $7_3$&5,13& 65\\ 
    $\mb{7_{37}}$ & $0_1$ & $5_2$ & $7_3$&7,13& 91\\ 
    $\mb{7_{38}}$ & 2x$0_1$ & $7_4$ & &1,15& 15\\ 
    $\mb{7_{39}}$ & 2x$0_1$ & $7_4$ & &1,15& 15\\ 
    $\mb{7_{40}}$ & $0_1$ & $3_1$ & $7_4$ &3,15& 45\\ 
    $\mb{7_{41}}$ & $0_1$ & $3_1$ & $7_4$ &3,15& 45\\ 
    $\mb{7_{42}}$ & $0_1$ & $5_2$ & $7_4$ &7,15& 105\\ 
    $\mb{7_{43}}$ & 2x$0_1$ & $7_5$ & &1,17& 17\\ 
    $\mb{7_{44}}$ & 2x$0_1$ & $7_5$ & &1,17& 17\\ 
    $\mb{7_{45}}$ & $0_1$ & $3_1$ & $7_5$ &3,17& 51\\ 
    $\mb{7_{46}}$ & $0_1$ & $3_1$ & $7_5$ &3,17& 51\\ 
    $\mb{7_{47}}$ & $0_1$ & $3_1$ & $7_5$ &3,17& 51\\ 
    $\mb{7_{48}}$ & $0_1$ & $5_1$ & $7_5$ &5,17& 85\\ 
    $\mb{7_{49}}$ & $0_1$ & $5_2$ & $7_5$ &7,17& 119\\ 
    $\mb{7_{50}}$ & 2x$0_1$ & $7_6$ & &1,19& 19\\ 
    $\mb{7_{51}}$ & 2x$0_1$ & $7_6$ & &1,19& 19\\ 
    $\mb{7_{52}}$ & 2x$0_1$ & $7_6$ & &1,19& 19\\ 
    $\mb{7_{53}}$ & 2x$0_1$ & $7_6$ & &1,19& 19\\ 
    $\mb{7_{54}}$ & 2x$0_1$ & $7_6$ & &1,19& 19\\ 
    $\mb{7_{55}}$ & $0_1$ & $3_1$ & $7_6$&3,19& 57\\ 
    $\mb{7_{56}}$ & $0_1$ & $3_1$ & $7_6$&3,19& 57\\ 
    $\mb{7_{57}}$ & $0_1$ & $4_1$ & $7_6$&5,19& 95\\ 
    $\mb{7_{58}}$ & $0_1$ & $5_2$ & $7_6$&7,19& 133\\ 
    $\mb{7_{59}}$ & 2x$0_1$ & $7_7$ & &1,21& 21\\ 
    $\mb{7_{60}}$ & 2x$0_1$ & $7_7$ & &1,21& 21\\ 
    $\mb{7_{61}}$ & 2x$0_1$ & $7_7$ & &1,21& 21\\ 
    $\mb{7_{62}}$ & 2x$0_1$ & $7_7$ & &1,21& 21\\ 
    $\mb{7_{63}}$ & 2x$0_1$ & $7_7$ & &1,21& 21\\ 
    $\mb{7_{64}}$ & $0_1$ & $3_1$ & $7_7$&3,21& 63\\ 
    $\mb{7_{65}}$ & $0_1$ & $4_1$ & $7_7$&5,21& 105\\ 
    \bottomrule

      \end{tabular}
 \end{multicols}
\end{center}

\label{det table}
 \end{table}

\bibliographystyle{alpha}
\bibliography{bibliography}

\end{document}